\documentclass[12pt]{amsart}
\usepackage{amssymb,latexsym}    
\usepackage{fullpage}

\input{xypic}
 \newcommand{\Mdef}[2]{\newcommand{#1}{\relax \ifmmode #2 \else $#2$\fi}}

\newcommand{\codim}{\mathrm{codim}}
\newcommand{\injdim}{\mathrm{injdim}}

\newcommand{\rank}{\mathrm{rank}}

\newcommand{\im}{\mathrm{im}}

\newcommand{\sm }{\wedge}

\newcommand{\tensor}{\otimes}

\newcommand{\Hom}{\mathrm{Hom}}

\newcommand{\Ext}{\mathrm{Ext}}
\Mdef{\bhom}{\mathbf{\hat{H}om}}

\Mdef{\Mod}{\mathrm{mod}}

\newcommand{\st}{\; | \;}

\newtheorem{thm}{Theorem}[section]
\newtheorem{lemma}[thm]{Lemma}
\newtheorem{prop}[thm]{Proposition}
\newtheorem{cor}[thm]{Corollary}

\theoremstyle{definition}

\newtheorem{defn}[thm]{Definition}

\newtheorem{example}[thm]{Example}

\newtheorem{remark}[thm]{Remark}

\newcommand{\qqed}{\qed \\[1ex]}
\renewenvironment{proof}[1][\hspace*{-.8ex}]{\noindent {\bf Proof #1:\;}}{\qqed}

\Mdef{\PH} {\Phi^H}
\Mdef{\PK} {\Phi^K}
\Mdef{\PL} {\Phi^L}
\Mdef{\PT} {\Phi^{\T}}

\Mdef{\ef}{E{\cF}_+}
\Mdef{\etf}{\widetilde{E}{\cF}}
\Mdef{\eg}{E{G}_+}
\Mdef{\etg}{\tilde{E}{G}}

\newcommand{\piA}{\pi^{\cA}}

\Mdef{\infl}{\mathrm{inf}}
\Mdef{\defl}{\mathrm{def}}
\Mdef{\res}{\mathrm{res}}
\Mdef{\ind}{\mathrm{ind}}
\Mdef{\coind}{\mathrm{coind}}

\Mdef{\univ}{\mathcal{U}}

\Mdef{\Fp}{\mathbb{F}_p}
\Mdef{\Zpinfty}{\Z /p^{\infty}}
\Mdef{\Zpadic}{\Z_p^{\wedge}}

\newcommand{\bi}{\begin{itemize}}
\newcommand{\be}{\begin{enumerate}}
\newcommand{\bc}{\begin{center}}
\newcommand{\bd}{\begin{description}}
\newcommand{\ei}{\end{itemize}}
\newcommand{\ee}{\end{enumerate}}
\newcommand{\ec}{\end{center}}
\newcommand{\ed}{\end{description}}

\newcommand{\adjunction}[4]{
\diagram
#1:#2 \rrto<0.7ex> &&
#3  \llto<0.7ex> :#4 
\enddiagram}
\newcommand{\lra}{\longrightarrow}
\newcommand{\lla}{\longleftarrow}

\newcommand{\Gspectra}{\mbox{$G$-{\bf spectra}}}

\Mdef{\we}{\mathbf{we}}
\Mdef{\fib}{\mathbf{fib}}
\Mdef{\cof}{\mathbf{cof}}
\Mdef{\BI}{\mathcal{BI}}

\newcommand{\ilim}{\mathop{ \mathop{\mathrm{lim}} \limits_\leftarrow} \nolimits}
\newcommand{\colim}{\mathop{  \mathop{\mathrm {lim}} \limits_\rightarrow} \nolimits}

\Mdef{\A}{\mathbb{A}}
\Mdef{\B}{\mathbb{B}}
\Mdef{\C}{\mathbb{C}}
\Mdef{\D}{\mathbb{D}}
\Mdef{\E}{\mathbb{E}}
\Mdef{\T}{\mathbb{T}}
\Mdef{\F}{\mathbb{F}}
\Mdef{\G}{\mathbb{G}}
\Mdef{\I}{\mathbb{I}}
\Mdef{\N}{\mathbb{N}}
\Mdef{\Q}{\mathbb{Q}}
\Mdef{\R}{\mathbb{R}}
\Mdef{\bbS}{\mathbb{S}}
\Mdef{\Z}{\mathbb{Z}}

\Mdef{\bA}{\mathbb{A}}
\Mdef{\bB}{\mathbb{B}}
\Mdef{\bC}{\mathbb{C}}
\Mdef{\bD}{\mathbb{D}}
\Mdef{\bE}{\mathbb{E}}
\Mdef{\bF}{\mathbb{F}}
\Mdef{\bG}{\mathbb{G}}
\Mdef{\bH}{\mathbb{H}}
\Mdef{\bI}{\mathbb{I}}
\Mdef{\bJ}{\mathbb{J}}
\Mdef{\bK}{\mathbb{K}}
\Mdef{\bL}{\mathbb{L}}
\Mdef{\bM}{\mathbb{M}}
\Mdef{\bN}{\mathbb{N}}
\Mdef{\bO}{\mathbb{O}}
\Mdef{\bP}{\mathbb{P}}
\Mdef{\bQ}{\mathbb{Q}}
\Mdef{\bR}{\mathbb{R}}
\Mdef{\bS}{\mathbb{S}}
\Mdef{\bT}{\mathbb{T}}
\Mdef{\bU}{\mathbb{U}}
\Mdef{\bV}{\mathbb{V}}
\Mdef{\bW}{\mathbb{W}}
\Mdef{\bX}{\mathbb{X}}
\Mdef{\bY}{\mathbb{Y}}
\Mdef{\bZ}{\mathbb{Z}}

\Mdef{\cA}{\mathcal{A}}
\Mdef{\cB}{\mathcal{B}}
\Mdef{\cC}{\mathcal{C}}
\Mdef{\mcD}{\mathcal{D}} 
\Mdef{\cE}{\mathcal{E}}
\Mdef{\cF}{\mathcal{F}}
\Mdef{\cG}{\mathcal{G}}
\Mdef{\mcH}{\mathcal{H}} 
\Mdef{\cI}{\mathcal{I}}
\Mdef{\cJ}{\mathcal{J}}
\Mdef{\cK}{\mathcal{K}}
\Mdef{\mcL}{\mathcal{L}}

\Mdef{\cM}{\mathcal{M}}
\Mdef{\cN}{\mathcal{N}}
\Mdef{\cO}{\mathcal{O}}
\Mdef{\cP}{\mathcal{P}}
\Mdef{\cQ}{\mathcal{Q}}
\Mdef{\mcR}{\mathcal{R}}
\Mdef{\cS}{\mathcal{S}}
\Mdef{\cT}{\mathcal{T}}
\Mdef{\cU}{\mathcal{U}}
\Mdef{\cV}{\mathcal{V}}
\Mdef{\cW}{\mathcal{W}}
\Mdef{\cX}{\mathcal{X}}
\Mdef{\cY}{\mathcal{Y}}
\Mdef{\cZ}{\mathcal{Z}}

\Mdef{\At}{\tilde{A}}
\Mdef{\Bt}{\tilde{B}}
\Mdef{\Ct}{\tilde{C}}
\Mdef{\Et}{\tilde{E}}
\Mdef{\Ht}{\tilde{H}}
\Mdef{\Kt}{\tilde{K}}
\Mdef{\Lt}{\tilde{L}}
\Mdef{\Mt}{\tilde{M}}
\Mdef{\Nt}{\tilde{N}}
\Mdef{\Pt}{\tilde{P}}

\Mdef{\tA}{\tilde{A}}
\Mdef{\tB}{\tilde{B}}
\Mdef{\tC}{\tilde{C}}
\Mdef{\tE}{\tilde{E}}
\Mdef{\tH}{\tilde{H}}
\Mdef{\tK}{\tilde{K}}
\Mdef{\tL}{\tilde{L}}
\Mdef{\tM}{\tilde{M}}
\Mdef{\tN}{\tilde{N}}
\Mdef{\tP}{\tilde{P}}

\Mdef{\ft}{\tilde{f}}
\Mdef{\xt}{\tilde{x}}
\Mdef{\yt}{\tilde{y}}

\Mdef{\Ab}{\overline{A}}
\Mdef{\Bb}{\overline{B}}
\Mdef{\Cb}{\overline{C}}
\Mdef{\Db}{\overline{D}}
\Mdef{\Eb}{\overline{E}}
\Mdef{\Fb}{\overline{F}}
\Mdef{\Gb}{\overline{G}}
\Mdef{\Hb}{\overline{H}}
\Mdef{\Ib}{\overline{I}}
\Mdef{\Jb}{\overline{J}}
\Mdef{\Kb}{\overline{K}}
\Mdef{\Lb}{\overline{L}}
\Mdef{\Mb}{\overline{M}}
\Mdef{\Nb}{\overline{N}}
\Mdef{\Ob}{\overline{O}}
\Mdef{\Pb}{\overline{P}}
\Mdef{\Qb}{\overline{Q}}
\Mdef{\Rb}{\overline{R}}
\Mdef{\Sb}{\overline{S}}
\Mdef{\Tb}{\overline{T}}
\Mdef{\Ub}{\overline{U}}
\Mdef{\Vb}{\overline{V}}
\Mdef{\Wb}{\overline{W}}
\Mdef{\Xb}{\overline{X}}
\Mdef{\Yb}{\overline{Y}}
\Mdef{\Zb}{\overline{Z}}

\Mdef{\db}{\overline{d}}
\Mdef{\hb}{\overline{h}}
\Mdef{\qb}{\overline{q}}
\Mdef{\rb}{\overline{r}}
\Mdef{\tb}{\overline{t}}
\Mdef{\ub}{\overline{u}}
\Mdef{\vb}{\overline{v}}

\Mdef{\hc}{\hat{c}}
\Mdef{\he}{\hat{e}}
\Mdef{\hf}{\hat{f}}
\Mdef{\hA}{\hat{A}}
\Mdef{\hH}{\hat{H}}
\Mdef{\hJ}{\hat{J}}
\Mdef{\hM}{\hat{M}}
\Mdef{\hP}{\hat{P}}
\Mdef{\hQ}{\hat{Q}}

\Mdef{\thetab}{\overline{\theta}}
\Mdef{\phib}{\overline{\phi}}

\Mdef{\uA}{\underline{A}}
\Mdef{\uB}{\underline{B}}
\Mdef{\uC}{\underline{C}}
\Mdef{\uD}{\underline{D}}

\Mdef{\bolda}{\mathbf{a}}
\Mdef{\boldb}{\mathbf{b}}
\Mdef{\boldD}{\mathbf{D}}

\Mdef{\fm}{\frak{m}}
\Mdef{\fp}{\frak{p}}

\Mdef{\eps}{\epsilon}

\newcommand{\Rtop}{R_{top}}
\newcommand{\Rttop}{\widetilde{R}_{top}}

\newcommand{\cOcF}{\cO_{\cF}}
\newcommand{\cOcFG}{\cO_{\cF /G}}
\newcommand{\cOcFH}{\cO_{\cF /H}}
\newcommand{\cOcFK}{\cO_{\cF /K}}
\newcommand{\cOcFL}{\cO_{\cF /L}}
\newcommand{\cOcFSj}{\cO_{\cF /S_j}}
\newcommand{\cOcFKSj}{\cO_{\cF /K\times S_j}}

\newcommand{\modcat}[1]{\mbox{$#1$-mod}}

\newcommand{\qcmodcat}[1]{\mbox{qc-$#1$-mod}}
\newcommand{\emodcat}[1]{\mbox{e-$#1$-mod}}
\newcommand{\qcemodcat}[1]{\mbox{qce-$#1$-mod}}

\newcommand{\modcatG}[1]{\mbox{$#1$-mod-$G$-spectra}}

\newcommand{\modcatGK}[1]{\mbox{$#1$-mod-$G/K$-spectra}}

\newcommand{\modcatGG}[1]{\mbox{$#1$-mod-spectra}}

\newcommand{\Rmod}{\modcat{R}}

\newcommand{\Rttopmod}{\modcatG{\Rttop}}
\newcommand{\iiRttopmod}{\modcatG{i_*i^*\Rttop}}

\newcommand{\naiveRttopmod}{\modcatG{i^*\Rttop}}

\newcommand{\Rtopmod}{\modcatGG{\Rtop}}

\newcommand{\RttopKmod}{\modcatG{\Rttop (K,1)}}
\newcommand{\PKDEFmod}{\modcatGK{\Phi^K\DH \efp}}

\newcommand{\DEFKmod}{\modcatGK{\DH \efkp}}

\newcommand{\afGKspectra}{\mbox{af-$G/K$-spectra}}

\newcommand{\efp}{E\cF_+}

\newcommand{\efkp}{E\cF/K_+}

\newcommand{\siftyV}[1]{S^{\infty V(#1)}}

\newcommand{\lr}[1]{\langle #1\rangle}

\newcommand{\connsub}{\mathbf{ConnSub(G)}}
\newcommand{\connquotG}{\mathbf{ConnQuot(G)}}

\newcommand{\dimC}{\mathrm{dim}_{\C}}

\newcommand{\cEi}{\cE^{-1}}

\renewcommand{\DH}{D}

\newcommand{\Deltat}{\tilde{\Delta}}

\newcommand{\bfD}{\mathbf{D}}

\newcommand{\qpG}{\mathrm{Q_2(G)}}
\newcommand{\qpGM}{\mathrm{Q_2(G/M)}}

\newcommand{\qcRmod}{\qcmodcat{R}}
\newcommand{\eRmod}{\emodcat{R}}
\newcommand{\eRGmod}{\emodcat{R_G}}

\newcommand{\eRGMmod}{\emodcat{R_{G/M}}}
\newcommand{\qceRmod}{\qcemodcat{R}}

\newcommand{\alphat}{\tilde{\alpha}}
\newcommand{\betat}{\tilde{\beta}}

\newcommand{\Cech}{\v{C}ech}
\newcommand{\CC}{\check{C}}
\newcommand{\SK}{K^{\bullet}_{\infty}}
\newcommand{\cAh}{\hat{\mathcal{A}}}
\newcommand{\Vbar}{\overline{V}}

\newcommand{\cAe}{\cA^e}

\begin{document}
\title{Rational torus-equivariant stable homotopy II: 
 algebra of the standard model}
\author{J.~P.~C.~Greenlees}
\address{Department of Pure Mathematics, The Hicks Building, 
Sheffield S3 7RH. UK.}
\email{j.greenlees@sheffield.ac.uk}

\date{}
\maketitle
\tableofcontents

\section{Introduction}

The purpose of this paper is to prove a number of purely algebraic
results about the category $\cA (G)$ constructed in \cite{tnq1}  to 
model rational
$G$-equivariant cohomology theories, where $G$ is a torus of rank 
$r\geq 0$. In more detail, the paper \cite{tnq1} introduced an abelian category $\cA (G)$, 
and a homology functor $\piA_* :\Gspectra \lra \cA (G)$, and showed
that they can be used to give an Adams spectral sequence
$$\Ext_{\cA (G)}^{*,*}(\piA_*(X), \piA_*(Y))\Rightarrow [X,Y]^G_*$$
convergent for any rational $G$-spectra $X$ and $Y$. Furthermore, the
category $\cA (G)$ was shown to be of injective dimension $\leq 2r$, so that
the spectral sequence converges in a finite number of steps. 

This is already be enough to motivate an algebraic 
study of the abelian category 
$\cA (G)$, but in joint work with Shipley \cite{tnq3}, it is shown that the 
Adams spectral sequence can be lifted to a Quillen equivalence
$$\Gspectra/\Q \simeq DG-\cA (G).$$
This means that $\cA (G)$ is not only a means for calculation but also it
is a means of construction. Accordingly, phenonmena discovered in 
$\cA (G)$ are realized in $G$-spectra.

The purpose of the present paper is to give two basic results about 
$\cA (G)$. Firstly, we show (Theorem \ref{thm:injdim})
that $\cA (G)$ has injective dimension precisely $r$. Secondly we
give constructions of certain torsion functors which allow us to 
construct certain right adjoints by first working in a larger category 
and then applying the torsion functor; the existence (but not the
construction) of the right adjoint is used in \cite{tnq3}. 
Along the way, we have an opportunity to prove a flatness result, to 
describe algebraic counterparts of some basic change of groups adjunctions, and 
to introduce terminology which fits well with the more elaborate structures used in 
\cite{tnq3}. 

It is intended to return to the study of $\cA (G)$ elsewhere, 
giving a more systematic study of $\cA (G)$,  highlighting its
similarities to categories of sheaves over projective varieties and 
explaining  the local-to-global principles alluded to in \cite{tnq1}.

\section{The algebraic model.}
\label{sec:Standard}
In this section we recall relevant results from \cite{tnq1} which 
constructs an abelian category $\cA (G)$ modelling the category 
of $G$-spectra and an Adams spectral sequence based on it. The
structures from that analysis will be relevant to much of what
we do here. 

\subsection{Definition of the category.}
First we must construct the the category $\cA (G)$.
For the purposes of this paper we view this as a category of modules
over a diagram of rings. 

The diagram of rings is modelled on the partially ordered set $\connsub$ 
of connected subgroups of $G$.  To start with we consider the single
ring
$$\cOcF =\prod_{F\in \cF }H^*(BG/F), $$
where the product is over the family $\cF$ of finite subgroups of $G$. 
To specify the value of the ring at a connected subgroup $K$,  
we use Euler classes: indeed if $V$ is a representation of $G$
we may defined $c(V) \in \cO_{\cF}$ by taking its components
$c(V)(F)=c_H(V^F) \in H^*(BG/F)$ to be classical Euler class
for ordinary homology.

The diagram  $\cO$ of rings is defined by 
$$\cO (K)=\cEi_K \cOcF$$
where $\cE_K =\{ c(V) \st V^K=0\} \subseteq \cOcF$ is the multiplicative
set of Euler classes of $K$-essential representations.

The category $\cA (G)$ is a category of modules $M$ over the diagram 
$\cO$ of rings. Thus the value $M(K)$ is a module over $\cEi_K\cOcF$, and if
$L\subseteq K$, the structure map 
$$\beta_L^K:M(L)\lra M(K)$$
is a map of modules over the map 
$$\cEi_L \cOcF \lra \cEi_K \cOcF$$
of rings. 

The category $\cA (G)$ is a certain non-full subcategory of the
category of $\cO$-modules. There are two requirements.
Firstly they must 
be {\em quasi-coherent}, in that they are determined by their 
value at the trivial subgroup $1$ by the formula 
$$M(K)=\cEi_K M(1). $$

The second condition involves the relation between $G$ and its quotients. 
Choosing a particular connected subgroup $K$, we consider the
relationship between the group
$G$ with the collection $\cF$ of its finite subgroups  
and the quotient group $G/K$ 
with the collection $\cF /K$ of its finite subgroups.  
For $G$ we have the ring $\cOcF$ and for $G/K$ we have 
the ring
$$\cOcFK =\prod_{\tK \in \cF /K}H^*(BG/\tK)$$
where we have identified finite subgroups of $G/K$ with 
their inverse images in $G$, i.e., with subgroups $\tK$ of $G$
having identity component $K$. There is an inflation map 
$$\infl: \cOcFK \lra \cOcF$$
whose $F$th component is the inflation map for the quotient $G/F\lra G/(FK)$.
The second condition is  that the object should be {\em extended}, in 
the sense that for each connected subgroup $K$ there is a specified isomorphism 
$$M(K)=\cEi_K \cOcF \otimes_{\cOcFK} \phi^K M$$
for some $\cOcFK$-module $\phi^KM$. These identifications should be
compatible in the evident way when we have inclusions of connected
subgroups. 

\subsection{Spheres.}
The only objects we will need to make explicit are spheres of virtual 
representations. We begin with spheres of  genuine representations. 

It is convenient to use suspensions, so we recall the definition. 
If $V$ is an  $n$-dimensional complex representation we divide
$\cF$ into $n+1$ sets $\cF_0, \cF_1, \cdots , \cF_n$ where 
$F\in \cF_i$ if $\dimC(V^F)=i$. All but one of these sets is 
finite. Now, for an 
$\cOcF$-module $M$, the suspension $\Sigma^VM$ is defined by 
breaking it into summands corresponding to the partition of finite
sets, and suspending by the appropriate amount on each one: 
$$\Sigma^VM=\bigoplus_{i=0}^n\Sigma^{2i}e_{\cF_i}M.$$

Turning to spheres, we write $S^V$ both for 
the usual one-point compactification and for the 
associated object of $\cA (G)$. To start with we have
$$\Phi^KS^V=S^{V^K}. $$
At the identity subgroup we have 
$$S^V(1)=\Sigma^VS^0(1)=\Sigma^V\cOcF, $$
this shows us that 
$$\phi^KS^V=\Sigma^{V^K}\cOcFK.$$

We would like to consider the fundamental class 
$$\iota_V=\Sigma^V1 \in \Sigma^V \cOcF =S^V(1). $$
However, we note that this is not a class in a single 
degree. It is a finite sum of homogeneous elements, namely 
the idempotent summands corresponding to the partition of $\cF$
into finite subgroups $F$ with $\dimC(V^F)$ constant. Bearing
this abuse of notation in mind, $\iota_V$ behaves like
a generator of the $\cOcF$-module $S^V(1)$ (its homogeneous 
summands generate).  We note that the element
$$1\tensor \iota_{V^K} \in \cEi_K\cOcF \tensor_{\cOcFK}\Sigma^{V^K} \cOcFK =S^V(K),  $$
acts as a generator in the same sense. 

The structure maps have the effect
$$e(V-V^K)\beta_1^K(\iota_V)=1\tensor \iota_{V^K},  $$
so that in general we have
$$\beta_L^K(\iota_{V^L})=e(V^L-V^K)^{-1}\tensor \iota_{V^K} \mbox{ in }
\cEi_K\cOcFL \tensor_{\cOcFK}\Sigma^{V^K}\cOcFK.$$

The contents of this section apply without essential change to 
the case when $V=V_0-V_1$  is a virtual representation. It is only 
necessary to use multiplicativity to define $e(V_0-V_1)=e(V_0)/e(V_1)$, 
and to note that the inverses of Euler classes make sense wherever they 
have been used. 
\subsection{Maps out of spheres.}
Next we need to understand maps out of spheres, so 
we consider a  map $\theta : S^V \lra Y$.

\begin{lemma}
\label{lem:mapsoutofSV} 
Maps  $\theta : S^V \lra Y$ correspond to systems of elements 
$x_K \in \Sigma^{-V^K}\phi^KY$ for all connected subgroups, 
with the property that if $L\subseteq K$  
$$e(V^L-V^K)\beta_L^K(x_L)=1\tensor x_K.$$
The correspondence is specified by 
$$\theta (K)(1\tensor \iota_{V^K})=1\tensor x_K.$$
\end{lemma}

\begin{proof}
Since the domain is the suspension of a free module, 
 maps $\Sigma^{V^K}\cOcFK \lra \phi^KY$ are uniquely specified
by elements $x_K$. For compatibility, note that whenever we have a containment
$L \subseteq K$ of connected subgroups we have a commutative
square
$$\begin{array}{ccc}
e(V^L-V^K)^{-1} \tensor \iota_{V^K} &\stackrel{\theta}\longmapsto 
& \beta_L^K(x_L)\\
\uparrow &&\uparrow\\
\iota_{V^L} &\stackrel{\theta}\longmapsto &x_L. 
\end{array}$$
\end{proof}

\begin{remark}
It is useful to be able to refer to patterns of this sort. Thus
a {\em footprint} of  $x\in Y(1)$ is given by the pattern of its images
under the basing maps. More precisely, it is a function defined
on connected subgroups, and if 
$$\beta_1^K(x)=\Sigma_i \lambda_i\tensor y_{K,i}, $$
the value of the footprint at $K$ is the set $\{ \lambda_i\}$ of elements
of $\cEi_K\cOcF$. 
Of course the expression is not unique, so an element will have many 
footprints, and it is instructive to consider what footprints look like and
when there are canonical footprints for elements. 
The footprint is somewhat analogous to the divisor of a function on an algebraic
variety.
\end{remark}

It is worth noting that maps out of spheres are determined by 
their value at 1. 

\begin{lemma}
A map $\theta$ is determined by the value
$\theta (\iota_V) \in Y(1)$. 
\end{lemma}

\begin{proof}
Note that $S^V(K)=\cEi_K \cOcF \tensor_{\cOcFK}\Sigma^{V^K}\cOcFK$. 
Under the restriction maps $e(V-V^K)\iota_V$ maps to 
$\iota_{V^K}$. Since $e(V-V^K)$ is a unit in the range,  
$\iota_V$ maps to $e(V-V^K)^{-1}\tensor \iota_{V^K}$. Accordingly, 
$$\theta (1\tensor \iota_{V^K})=e(V-V^K)^{-1}\tensor 
\beta (\theta (\iota_V)).$$ 
\end{proof}

It is convenient to write $Y(V)$ for the image of the evaluation map 
$$\Hom (S^{-V} ,Y)\lra Y(1) .$$

\section{Inflation maps and localization maps.}
\label{sec:PKDEFpvsDEFKp} 

In this section we discuss two maps between rings that arise
in the structure of objects of $\cA (G)$. When we form 
$\cEi_K \cOcF \tensor_{\cOcFK}N$ we are concerned with the 
$\cOcFK$-module structure of $\cEi_K \cOcF$. We therefore
need to discuss the inflation map
$$\infl =\infl_{G/K}^G:  \cOcFK \lra \cOcF$$
from $G/K$ to $G$, and then consider the localization.

\begin{prop}
\label{prop:splitmono}
For any connected subgroup $K$, both of the maps
$$\cOcFK \stackrel{\infl}\lra \cOcF \stackrel{l}\lra \cEi_K \cOcF$$
are split monomorphisms of $\cOcFK$-modules, and hence in 
particular they are flat. 
\end{prop}

We will treat the two maps in the following two subsections. 

\subsection{Inflation maps}
We begin by describing the maps in more detail.
For this we need  the map $q:\cF \lra \cF/K$ taking the 
image of a finite subgroup of $G$ in $G/K$. First note that the 
condition $q(F)=\tK /K$ amounts to $FK=\tK$. In particular, 
 $\tK \supseteq F$ so that there is a map $G/F \lra G/\tK$ inducing
inflation 
$$\infl: H^*(BG/\tK ) \lra H^*(BG/F). $$
Indeed, this makes $H^*(BG/F)$ into a polynomial algebra over
$H^*(BG/\tK)$, and hence it free as a  module.
We can find free module generators by choosing a splitting 
$G/F \lra K/(K \cap F)$ of the inclusion, and using the image of 
$H^*(BK/(K\cap F))$.

Since we are working over the rationals, this is isomorphic
to the case $F=1$, where notation is less cluttered. 
We have a short exact sequence
$$0\lra H^*(BG/K)\lra H^*(BG)\lra H^*(BK)\lra 0$$
of algebras, and where we may choose a splitting $G\cong G/K \times K$
to show 
$$H^*(BG)\cong H^*(BG/K)\tensor H^*(BK). $$

Taking all finite subgroups with $q(F)=\tK$ we obtain a map 
$$\Deltat_{\tK} : H^*(BG/\tK ) \lra \prod_{FK=\tK} H^*(BG/F). $$
The codomain is a product of free modules over $H^*(BG/\tK)$.
The entire map $q^*: \cOcFK \lra \cOcF$ is the product 
of the maps $\Deltat_{\tK}$ over subgroups
$\tK$ with identity component $K$.  The codomain is a product of 
projective modules over $\cOcFK$.

\begin{lemma}
For each subgroup $\tK$ with identity component
$K$, the map $\Deltat_{\tK}$ is a split monomorphism of free
$H^*(BG/\tK)$-modules, and in particular it is flat. 
\end{lemma}

\begin{proof}
To start with, we note that since any vector space has a basis, 
the product $\prod_{i}\Q$ is actually a sum of copies of $\Q$. 
Now consider the polynomial ring $P=H^*(BG/\tK)$.
Since $P$ is finite dimensional in each degree and cohomologically 
bounded below, the natural map 
$$\left[ \prod_{i}\Sigma^{n_i} \Q \right] \tensor_{\Q} P \lra 
\prod_{i} \Sigma^{n_i}P $$
is an isomorphism provided the suspensions $n_i$ are all cohomologically 
positive and with finitely many in each degree. Each of the $P$-modules
$H^*(BG/F)$ can be written in the form 
$$H^*(BG/F)=\bigoplus_j \Sigma^{n_j} P\cong \prod_j \Sigma^{n_j} P. $$
so we see that $\prod_{FK =\tK} H^*(BG/F)$ is a free $P$-module.

\end{proof}

\subsection{Localization maps}
We now turn to the localization map $\cOcF \lra \cEi_K \cOcF$, viewing it
as a map of $\cOcFK$-modules via inflation. 

If $W$ is any representation with $W^G=0$ the suspension $\Sigma^w \cOcF$
is the projective $\cOcF$-module obtained by suspending the $F$th factor
by $\dimC(W^F)$. We may compose $q^*$ with the map
$$\cOcF \lra \Sigma^w \cOcF, $$
so that the codomain is again a product of projective $\cOcFK$-modules.

If $W_1 \subseteq W_2$ are two representations with $W_1^G=W_2^G=0$ 
the inclusion $\Sigma^{w_1}\cOcF \lra \Sigma^{w_2}\cOcF$ is multiplication 
by $c^{\dimC(W_2^F)-\dimC(W_1^F)}$ on the $F$th factor. If $W_1^K=W_2^K=0$,
using the tensor product decomposition $H^*(BG)=H^*(BG/K)\tensor H^*(BK)$ 
we see that  the resulting map $H^*(BG)\lra \Sigma^{w_2-w_1}H^*(BG)$ is 
split mono as a map of free $H^*(BG/K)$-modules. Allowing quotients by 
finite subgroups and combining these, we see that
$\Sigma^{w_1}\cOcF \lra \Sigma^{w_2}\cOcF$ 
is split mono as a map of $\cOcFK$-modules. 
Passing to limits over representations $W$ with $W^K=0$, we 
see that
$$\cOcF \lra \cEi_K \cOcF$$ 
is split mono as a map of $\cOcFK$-modules as required. 

\section{Homological dimension.}

The purpose of this section is to establish the exact injective dimension 
of $\cA (G)$. 

\begin{thm}
\label{thm:injdim}
The category $\cA(G)$ has injective dimension $r$.
\end{thm}

Since the category of torsion $H^*(BG)$-modules has injective dimension $r$
and sits as an abelian subcategory of $\cA (G)$, 
it follows that any object of $\cA (G)$ concentrated at 1 has
injective dimension $\leq r$, and the torsion module $\Q$ shows that
there are objects of injective dimension $r$.  It remains to show that arbitrary objects of 
$\cA (G)$  have injective dimension $\leq r$.

In \cite[5.3]{tnq1} we showed that $\cA (G)$ is of injective dimension $\leq 2r$, 
and deferred the proof of the exact bound. The proof  is reminiscent 
of arguments with sheaves on a variety: although $\cA (G)$ does not have enough 
projectives, the spheres provide a class of objects analagous to twists of the 
structure sheaf.  One expects every object to be a quotient of 
extensions of these, whilst on the other hand, we can  show they 
have injective dimension $\leq r$ by explicit construction.

\begin{remark}
One can also imagine a proof as in the rank 1 case \cite[5.5.2]{s1q},
where one uses the fact
that not only is the  category of torsion modules over $\Q [c]$ of 
injective dimension 1, but so is the category of all modules. The 
author has not managed to implement this argument in general, because
it involves isolating certain properties of injective resolutions. 
On the other hand it would give valuable insights.
For example at the level of objects supported in 
dimension 0, one would need to identify a suitable 
subcategory of all modules which includes all the $H^*(BG)$-modules 
that occur as $e_1M(1)$ for an object of $\cAe (G)$ and has products. 
One would need to understand injective resolutions in this category, and it would
be convenient if  the modules $\cEi_K H^*(BG)\tensor_{H^*(BG/K)}
H_*(BG/K)$ were injective in the subcategory. 
\end{remark}

\subsection{Injective dimension.}

We first explain the strategy, which is to find a large enough 
class of objects known to have injective dimension $\leq r$.

We say that $M$ is {\em detected in dimension $d$} if $M(H)=0$ 
for $H$ of dimension $>d$ and if $L \subseteq K$ with $K$ of
dimension $d$ then $M(L)\lra M(K)$ is monomorphism.

\begin{defn}
\label{defn:ide} 
We say that a set of objects  $\cP$ forms a set of 
{\em injective dimension estimators} if the following conditions
hold
\begin{enumerate}
\item Every object of $\cP$ has injective dimension $\leq r$
\item If $X$ is an object of $\cA (G)$ detected in dimension 
$d$, there is a map 
$$\bigoplus_iP_i \lra X$$
where $P_i$ is an object of $\cP$ and the cokernel has support in 
dimension $\leq d-1$. 
\end{enumerate}
\end{defn}

The existence of $\cP$ is enough to give a proof of Theorem \ref{thm:injdim}.

\begin{lemma}
If there is a set $\cP$ of injective dimension estimators then 
any qce module is of injective dimension $\leq r$.
\end{lemma}

\begin{proof}
To start with, we show by induction on $d$ that if $M$ has support of 
dimension $\leq d$ then $\injdim (M)\leq r$. 

If $d=0$ then $M$ is a {\em sum} over finite subgroups $F$ of torsion modules 
$M_F$ over $H^*(BG/F)$, and where each $\cEi_HM_F=0$ for all connected 
subgroups $H$ \cite[4.5]{tnq1}. It follows that $M_F$ can be
embedded in a sum of copies of realizable injective modules $H_*(BG/F)$.
This establishes the base of the induction. The case when $d=r$ 
will give the statement of the lemma.

 We now suppose $d>0$ and our inductive hypothesis is that
 all modules with support of dimension $<d$  have injective 
dimension $\leq r$. 
For the inductive step we have a subsidiary induction. 
We suppose by downwards induction on $s$ that it has been shown that all modules
supported in dimension $\leq d$ have injective dimension $\leq s$.  
This holds for $s=2r $ by \cite[5.3]{tnq1}.
We show that if  $s > r$ then $s$ can be reduced, so that in at most $r$
steps we will have obtained the required estimate. 

Suppose then that $M'$ is supported in dimensions $\leq d$. 
First we observe that 
it suffices to deal with the quotient $M$ {\em detected} in dimension 
$d$ ($M$ may be constructed as the image of $M'$ in the sum of 
modules $f_H(M(H))$ in the resolution in \cite[5.3]{tnq1}). Indeed, we 
have an exact sequence $0\lra T\lra M' \lra M \lra 0$. Since $T$ 
is supported in dimension $\leq d-1$, it has injective dimension $\leq r$
by induction, so if $M$ has injective dimension $\leq s$ then so does $M'$. 

Since $\cP$ is  a set of injective dimension estimators, we may 
construct a map $P \lra M$, with $P$ a sum
of elements of $\cP$ so that the cokernel supported in dimension 
$\leq d-1$. Since sums of realizable injectives are injective by 
\cite[5.2]{tnq1}, we find $\injdim(P)\leq r$.
Factorizing $p$ gives two short exact sequences
$$0\lra K \lra P \lra \Pb \lra 0 \mbox{ and } 0\lra \Pb \lra M \lra \Mb 
\lra 0.$$
By the subsidiary induction $\injdim(K)\leq s$, and since
 $\injdim(P)\leq r\leq s-1$ by hypothesis, the first short exact sequence
shows $\injdim (\Pb)\leq s-1$. Turning to the second short exact sequence, 
since $\Mb$ is supported in dimension $\leq d-1$, it has injective dimension 
$\leq r \leq s-1$, and we deduce $\injdim (M)\leq s-1$ as required.
\end{proof}

It remains to show that a set of injective dimension estimators exists. 
We will find a large enough class to reduce dimension of support (i.e., 
to give Condition 2) in Subsection \ref{subsec:sufficiency}, and then 
in Subsection \ref{subsec:injdimspheres} show that they have 
injective dimension $\leq r$ (i.e., that they satisfy Condition 1).

\subsection{Sufficiency of spheres.}
\label{subsec:sufficiency}
We show that there are enough virtual spheres to satisfy the second condition for
a set of injective dimension estimators (\ref{defn:ide}).

\begin{prop}
If  $M$ is detected in dimension $d$
then there is  a map $P \lra M$ and $P$ a wedge of spheres 
with cokernel supported in dimension $\leq d-1$. 
\end{prop}

\begin{proof}
It suffices to
show that for any dimension $d$ subgroup $H$ and any
$x \in \phi^HM$ we can hit $e(U)\tensor x$ for some $U$ with 
$U^H=0$. By Lemma \ref{lem:mapsoutofSV} it suffices to argue 
that $x=x_H$ is part of a family of elements $x_K \in \phi^KM$ so that
for $K\subseteq H$ we have $\beta_1^K(x_1)=e(U/U^K)\tensor x_K$.

To start with we find $U$ and $x_1$. Indeed, the cokernel 
of $M(1) \lra M(H)$ is $\cE_H$-torsion, 
so we may find $U$ with $U^H=0$ and $x_1\in M(1)$ with $\beta_1^H(x_1)=e(U) 
\tensor x$. This automatically gives 
$\beta_1^K(x_1)$, but we must argue that this takes the required form.

For notational simplicity we write $U=\Ub \oplus U'$ where $\Ub=U^K$
and $(U')^K=0$.
Now consider $K \subseteq H$,  and use the fact that the cokernel of  
$$\phi^K M\lra \cEi_{H/K}\phi^KM =\cEi_{H/K}\cOcFK \tensor_{\cOcFH}\phi^HM$$
is $\cE_{H/K}$-torsion. This shows that there is a representation 
$\Vb$ of $G/K$ with $\Vb^H=0$ and $y_K\in \phi^KM$ with 
$$\beta_K^H(y_K)=e(\Vb)e(\Ub)\tensor x_H. $$ 
Next, we use the fact that the cokernel of $M(1) \lra M(K)$ is $\cE_K$-torsion
to choose $x_1'\in M(1) $ and a representation $V'$ with $(V')^K=0$ with
$$\beta_1^K(x_1')=e(U')e(V')\tensor y_K. $$ 
Accordingly 
$$\beta_1^H(x_1')=e(U')e(\Ub)e(V')e(\Vb) \tensor x_H. $$
Since the structure maps are monomorphic, we conclude $e(\Vb)e(V')x_1=x_1'$, 
and applying $\beta_1^K$ we obtain
$$e(V')e(\Vb)\beta_1^K(x_1)=\beta_1^K(x_1')=e(V')e(U')\tensor y_K.$$
Cancelling $e(V')$ (which is a unit since $(V')^K=0$), we find
$$e(\Vb)\beta_1^K(x_1)=e(U')\tensor y_K. $$
Now $e(\Vb)$ is inflated from $\cOcFK$, so that if 
$\beta_1^K(x_1)=\sum_i \lambda_i\tensor x_{K,i}$ the left hand side is
$\sum_i \lambda_i \tensor e(\Vb)x_{K,i}$. Finally, we wish to conclude
$$\beta_1^K(x_1)=e(U')\tensor \hat{y}_K$$
as required, where $e(\Vb)\hat{y}_K=y_K$. For this we use the fact that the
square
$$\begin{array}{ccc}
\phi^KM&\stackrel{e(\Vb)}\lra &\phi^KM\\
\downarrow &&\downarrow\\
\cEi_{K}\cOcF\tensor_{\cOcFK}\phi^KM
&\stackrel{1\tensor e(\Vb)}\lra &\cEi_{K}\cOcF\tensor_{\cOcFK}\phi^KM
\end{array}$$
is a pullback. This in turn follows since $\cOcFK\lra \cEi_{K}\cOcF$ is 
a split monomorphisms by Proposition \ref{prop:splitmono}. 
Indeed, we have the following elementary lemma 
\begin{lemma}
If $M\lra M'$ is a monomorphism of $R$-modules and 
$R\lra L$ is a split monomorphism of $R$-modules, then 
$$\begin{array}{ccc}
M&\lra &M'\\
\downarrow &&\downarrow\\
L\tensor_RM&\lra &L\tensor_R M'
\end{array}$$
is a pullback.\qqed
\end{lemma}

This completes the proof that there are sufficiently many virtual spheres. 
\end{proof}

\subsection{Injective dimension of spheres.}
\label{subsec:injdimspheres}

Finally we show that spheres satisfy the first condition for
a set of injective dimension estimators (\ref{defn:ide}).

\begin{prop}
Any sphere is of injective dimension $\leq r$.
\end{prop}

\begin{proof}
In effect we will show that the Cousin complex gives an injective
resolution. Although this is a purely algebraic result, the 
complex is realizable. Indeed, the resolution for $S^0$
is realizable by maps of spaces, and resolutions of other speheres
are obtained by suspension. 

In fact these resolutions are discussed in \cite[Section 12]{tnq1}, 
and the resolution of $S^0$ is the special case of \cite[Proposition 
12.3]{tnq1} for the family of all subgroups. More precisely, it is stated
that a certain sequence of spaces
$$S^0\lra \bigvee_{H \in \cF_{r-1}}E\lr{H}
\lra \bigvee_{H \in \cF_{r-2}}E\lr{H}
\lra \cdots \lra \bigvee_{H \in \cF_0}E\lr{H}$$
induces an exact sequence in $\piA_*$. Here $\cF_i$ denotes the set of 
subgroups of dimension $i$, and $\piA_*(E\lr{H})$ is injective for 
all subgroups $H$ by \cite[10.2]{tnq1}. Each of the maps comes from a
 cofibre sequence
$$E(\cF_{< d})_+\lra E(\cF_{\leq d})_+\lra \bigvee_{H\in \cF_d}E\lr{H}, $$
and the exactness of the resolution is proved by showing this induces
a {\em short} exact sequence in $\piA_*$. Indeed, we observe that the 
first of the displayed maps in the cofibre sequence is necessarily 
zero in $\piA_*$ because $\piA_*(E(\cF_{< d})_+)$ is supported in codimension 
$<d$ whereas $\piA_*(E(\cF_{\leq d})_+)$ is detected in codimension $d$. 

It remains to argue that the suspension of the sequence
is still exact and still consists of injectives. Exactness
follows, because suspension 
preserves being supported in dimension $<d$ and detection in dimension $d$. 
Suspension preserves injectives since 
 $\Sigma^V E\lr{H}\simeq \Sigma^{|V^H|} E\lr{H}$.  
\end{proof}

\section{A larger diagram of rings and  modules.}

We will now reformulate the definition of $\cA (G)$ slightly. In effect
we are just making some of the structure more explicit. This is convenient
for certain constructions, and also ensures the notation is consistent with 
that of the topological situation of \cite{tnq3}, 
where the additional notation is essential. 

\subsection{The diagram of connected quotients.}
The structure of our model is that in $\cA (G)$ we include a 
model for $K$-fixed point objects whenever $K$ is a non-trivial connected
subgroup. This leads us to consider the poset $\connsub$, 
 Because the subgroup $K$ indexes $G/K$-equivariant information
it is clearer to rename the objects and consider the poset
 $\connquotG$ of quotients of $G$ by connected subgroups, where the maps are
the quotient maps. 

Because this is fundamental, we standardize  the display so that 
$\connquotG$ is arranged horizontally (i.e., in the $x$-direction), 
with quotients decreasing in size 
from $G/1$ at the left to $G/G$ at the right. When $G$ 
is the circle $\connquotG =\{ G/1\lra G/G\}$ but when $G$ is a torus of rank 
$\geq 2$ it has a unique initial object $1$, a unique terminal object $G$, 
but infinitely many objects at every other level. This makes it harder to 
draw $\connquotG$, so we will usually draw a diagram by choosing one or two 
representatives from each level, so that in rank $r$, the diagram 
$\connsub$ is illustrated by 
$$G/1 \lra G/H_1\lra G/H_2\lra \cdots \lra G/H_{r-1} \lra G/G$$
where $H_i$ is an $i$-dimensional subtorus of $G$. 

\subsection{The diagram of quotient pairs.}
The following diagram is what is needed to index the relevant information
in our context. 

\begin{defn}
The diagram $\qpG$ of {\em quotient pairs} of $G$ is the partially ordered
set with objects $(G/K)_{G/L}$ for $L\subseteq K \subseteq G$, and with two 
types of morphisms. The {\em horizontal} morphisms
$$h_K^H: (G/K)_{G/L} \lra (G/H)_{G/L} \mbox{ for } L\subseteq K \subseteq H \subseteq G$$
and the {\em vertical} morphisms
$$v_L^K: (G/H)_{G/K} \lra (G/H)_{G/L} \mbox{ for } 
L\subseteq K \subseteq H \subseteq G.$$
We will refer to the terms $(G/H)_{G/L}$ with $\rank (G/L)=d$ as the {\em 
rank $d$ row}, and to those with $\rank (G/H)=d$ as the {\em 
rank $d$ column}. The terms $G/H=(G/H)_{G/1}$ form the {\em bottom row} and the 
terms $(G/L)_{G/L}$ form the {\em leading diagonal}. 
\end{defn}

\begin{remark}
The notation is generally chosen to be compatible with the slightly 
larger diagrams that are required in \cite{tnq3}. However, we have
simplified it slightly:  the poset $\qpG$ is
the localization-inflation diagram  $\mathbf{LI}(G)$ of \cite{tnq3},
and the object $(G/K)_{G/L} $ is there denoted $(G/K,c)_{G/L}$.
\end{remark}

By way of illustration, suppose $G$ is of rank $2$. 
The part of a diagram $R : \qpG \lra \C$ including only one circle 
subgroup $K$ would then take the form  
$$
 \diagram
           &                    &R(G/G)_{G/G}\dto\\
           &R(G/K)_{G/K}\rto \dto &R(G/G)_{G/K}\dto\\
R(G/1)_{G/1}\rto &R(G/K)_{G/1}\rto      &R(G/G)_{G/1}\\
\enddiagram $$
We think of the bottom row as consisting of the basic information, 
and the higher rows as providing additional structure. The omission of brackets to 
write $R(G/K)_{G/L}=R((G/K)_{G/L})$ is convenient and follows conventions
common in equivariant topology. 

\subsection{The structure ring.}
One particular diagram will be of special significance for us. 
\begin{defn}
The structure diagram for $G$ is the diagram of rings defined by 
$$R(G/K)_{G/L}:=\cEi_{K/L} \cOcFL . $$
Since $V^K=0$ implies $V^H=0$,  we see that 
$\cE_{H/L}\supseteq \cE_{K/L}$,  so it is legitimate to take the 
horizontal maps to be localizations
$$h_K^H: \cEi_{K/L} \cOcFL \lra \cEi_{H/L} \cOcFL . $$
To define the  vertical maps,  we begin with the inflation map 
$\infl_{G/K}^{G/L}: \cOcFK \lra \cOcFL$, and then observe that 
if $V$ is a representation of $G/K$ with $V^H=0$, it may be regarded
as a representation of $G/L$, and Euler classes correspond in the sense
that  $\infl(e_{G/K}(V))=e_{G/L}(V)$. We therefore obtain a map  
$$v_K^H: \cEi_{H/K}\cOcFK \lra \cEi_{H/L}\cOcFL. $$
\end{defn}

Illustrating this for a group $G$ of rank 2 in the usual way, we obtain
$$\diagram
           &                    &\cOcFG    \dto\\
           &\cOcFK \rto \dto &\cEi_{G/K}\cOcFK\dto\\
\cOcF \rto &\cEi_K \cOcF \rto      &\cEi_G    \cOcF\\
\enddiagram
$$
At the top right, of course $\cOcFG=\Q$, but clarifies the formalism
to use the more complicated notation. 

\subsection{Diagrams of rings and modules.}
Given a diagram shape $\bfD$, we may consider a diagram $R:\bfD \lra \C$
of rings in a category $\C$. We may then consider the category of $R$-modules.
These will be diagrams $M: \bfD \lra \C$ in which $M(x)$ is an $R(x)$-module
and for every morphism $a: x\lra y$ in $\bfD$, the map $M(a): M(x) \lra M(y)$ 
is a module map over the ring map $R(a): R(x) \lra R(y)$. 

It is sometimes convenient to use extension of scalars to obtain 
the map 
$$\widetilde{M}(a): R(y)\tensor_{R(x)} M(x) =R(a)_*M(x) \lra M(y) $$
of modules over the single ring $R(y)$.

\subsection{The  category of $R$-modules.}
In discussing modules, we need to refer to the structure maps for rings, 
so for an $R$-module $M$, if $L\subseteq K \subseteq H\subseteq G$, 
we generically write 
$$\alpha_K^H: M(G/H)_{G/K}\lra M( G/H)_{G/L}$$ 
for the vertical map, and 
$$\alphat_K^H: \cEi_{H/L}\cOcFL \tensor_{\cOcFK}M(G/H)_{G/K}
=(v_K^H)_*M(G/H)_{G/K} \lra M(G/H)_{G/L}$$ 
for the associated map of $\cOcFL$-modules. Similarly,
we generically write 
$$\beta_K^H: M(G/K)_{G/L}\lra M(G/H)_{G/L}$$ 
for the horizontal map, and 
$$\betat_K^H: \cEi_{H/L}M(G/K)_{G/L}
=(h_K^H)_*M(G/K)_{G/L} \lra M( G/H)_{G/L}$$ 
for the associated map of $\cEi_{H/L}\cOcFL$-modules, which 
we refer to as the {\em basing map} after \cite{s1q}. 

In our case the horizontal maps are simply localizations, so all maps
in the $G/L$-row  can reasonably be viewed as $\cOcFL$-module maps. 
On the other hand, the vertical maps increase the size of the rings, so 
it is convenient to replace the original diagram by the diagram in which 
all vertical maps in the $G/L$ column have had scalars extended to 
$\cEi_{L}\cOcF$. We refer to this as the $\alphat$-diagram, and think of 
it as a diagram of $\cOcF$-modules. 

\begin{defn}
If $M$ is an $R$-module, we say that $M$ is {\em extended} if
whenever $L\subseteq K \subseteq H$ the vertical map  $\alpha_K^H$
is an extension of scalars along $v_K^H:\cEi_{H/K}\cOcFK \lra 
\cEi_{H/L}\cOcFL,  $ which is to say that 
$$\alphat_K^H: \cEi_{H/L}\cOcFL \tensor_{\cOcFK} M(G/H)_{G/K}
\stackrel{\cong}\lra M(G/H)_{G/L}$$ 
is an isomorphism of $\cEi_{H/L}\cOcFL $-modules. 

If $M$ is an $R$-module, we say that $M$ is 
{\em quasi-coherent} if
whenever $L\subseteq K \subseteq H$ the horizontal map 
$\beta_K^H$ is an extension of scalars along $h_K^H:\cEi_{K/L}\cOcFL
\lra \cEi_{H/L}\cOcFL$, which is to say that 
$$\betat_K^H: \cEi_{H/L}M(G/K)_{G/L}\stackrel{\cong}\lra M(G/H)_{G/L}$$ 
is an isomorphism. 

We write $\qcRmod, \eRmod$ and $\qceRmod$ for the full subcategories
of  $R$-modules with the indicated properties. 
\end{defn}

Next observe that the most significant part of the information in an 
extended object is displayed in its restriction to the leading diagonal. 
For example in our rank 2 example they take the form
$$\diagram
           &                    & M(G/G)_{G/G}    \dto\\
           &M(G/K)_{G/K} \rto \dto &\cEi_{G/K}\cOcFK\tensor_{\cOcFG} M(G/G)_{G/G} \dto\\
M(G/1)_{G/1} \rto &\cEi_K \cOcF \tensor_{\cOcFK}M(G/K)_{G/K} \rto      
&\cEi_G    \cOcF\tensor_{\cOcFG}M(G/G)_{G/G}\\
\enddiagram
$$
We will typically abbreviate such a diagram by just writing the final 
row and abbreviating $M_{\phi}(G/K)=M(G/K)_{G/K}$:
$$\diagram
M_{\phi}(G/1) \rto &\cEi_K \cOcF \tensor_{\cOcFK}M_{\phi}(G/K) \rto      
&\cEi_G    \cOcF\tensor_{\cOcFG}M_{\phi}(G/G), \\
\enddiagram
$$
leaving it implicit that the particular decomposition as a tensor product
is part of the structure.

\subsection{The category $\protect \cA (G)$ as a diagram of modules over quotient pairs.}
Having this language allows us a convenient way to encode the information in the 
category $\cA (G)$. Indeed, there is a functor
$$i: \cA (G) \lra \Rmod$$
defined by 
$$i (M) (G/K)_{G/L}:=\cEi_{K/L}\phi^LM . $$

It is straightforward to encode the quasi-coherence condition of $\cA (G)$
in the horizontal maps and the extendedness in the vertical maps.

\begin{lemma}
The functor $i$ takes quasicoherent objects of $\cA (G)$ 
to quasicoherent objects of $\Rmod$ and extended objects of $\cA (G)$ 
to extended objects of $\Rmod$. The category $\cA (G)$ is isomorphic
to the category of qce $R$-modules. \qqed
\end{lemma}

In view of this, we henceforth view $i$ as being the inclusion of qce $R$-modules
in all $R$-modules, which may be factored as
$$\cA (G)= \qcemodcat{R}\stackrel{j}\lra \emodcat{R}\stackrel{k}\lra
\modcat{R}. $$
We will construct right adjoints $\Gamma_h$ to $j$ in Section \ref{sec:Gammah} and
$\Gamma_v=k^{!}$ to $k$ in Section \ref{sec:Gammav}. The subscripts
$h$ and $v$ refer to the fact that either the horizontal or vertical structure
maps have been forced to be extensions of scalars. Combining these, we obtain
the right adjoint
$$\Gamma =\Gamma_h\Gamma_v$$
to $i$. 

\section{Change of groups functors}

The two most important change of groups on spaces are fixed point
functors and inflation functors. We describe counterparts to these in 
our algebraic context.  

\subsection{The functor $\Phi^K$.}

We begin with the counterpart of the geometric fixed point functor, 
which is the natural extension of the fixed point space functor to spectra. 
Indeed, the model $\cA (G)$ is based around the geometric fixed point 
funcctor, so the algebraic counterpart  is painless to define and 
its good properties are built into the model.

For emphasis we write the ambient group as a subscript, so that 
$R_G$ denotes the $\qpG$-diagram of rings. The fixed point functor is
is defined on  extended $R_G$-modules. 

\begin{defn}
The functor $\Phi^M : \eRGmod \lra \eRGMmod$ is defined by restricting
a module $X$ from the diagram $\qpG$ to the subdiagram $\qpGM$.
In other words, it is defined by the formula
$$(\Phi^M X)(\Gb/\Hb)_{\Gb/\Kb}=X(G/H)_{G/K}$$
where $M\subseteq K \subseteq H \subseteq G$ and where bars denote 
the image in $G/M$.
\end{defn}

The following property is built in to the definition of qce $R_G$-modules.

\begin{lemma}
The functor $\Phi^M$ takes quasi-coherent modules to quasi-coherent modules.
\qqed
\end{lemma}

Unfortunately, $\Phi^M$ is not a right adjoint, but we will see that there
is a convenient substitute.

\subsection{Inflation}

Inflation involves filling in the diagram with extensions of scalars. 

\begin{defn}
The inflation functor $\infl=\infl_{G/M}^G: \eRGMmod \lra \eRGmod$
is defined by extension of scalars. To describe this in  more detail, 
we do not suppose that $M$ is contained in any of the other subgroups
as we would usually do, but we will assume $L \subseteq K \subseteq H
\subseteq G$.  For an $R_{G/M}$-module
$Y$ the $R_G$-module $\infl Y$ is defined on the leading diagonal via 
$$(\infl Y)(G/H)_{G/H}=Y(\Gb/\Hb)_{\Gb/\Hb}$$
where $\Hb$ denotes the image of $H$ in $\Gb =G/M$. The remaining 
entries are given by extending scalars along the vertical maps. 
The horizontal 
$$h_K^H: \infl Y (G/K)_{G/M} \lra \infl Y (G/H)_{G/M}$$
is obtained from 
$$h_K^H:  Y (\Gb/\Kb)_{\Gb/\Kb} \lra  Y (\Gb/\Hb)_{\Gb/\Kb}$$
by extending scalars along the inflation $\Gb/\Kb \lra G/M$. 
\end{defn}

From the definition, good behaviour on quasi-coherent modules is clear. 

\begin{lemma}
The functor $\infl$ takes quasi-coherent modules to quasi-coherent
modules. \qqed
\end{lemma}

\subsection{A substitute for a left adjoint.}
There is no left adjoint to the geometric fixed point functor, but
 we have the following description which fulfils a similar purpose.
For spectra, there is a familiar connection between representations and
geometric fixed points extending the one for spaces. For this, we write
$$\siftyV{M}=\bigcup_{V^M=0}S^V. $$ 
The statement for spectra is that 
$$[A,\Phi^MX]^{G/M}=[\infl A, X \sm \siftyV{M}]^G$$
or that Lewis-May and geometric fixed points are related by 
$$\Phi^MX \simeq (X\sm \siftyV{M})^M.$$
The counterpart to this in algebra
is built in to the category of modules at the level 
of the underlying abelian category.

\begin{prop}
\label{prop:leftPhi}
If $X$ and $A$ are quasi-coherent, 
there is an isomorphism 
$$\Hom_{R_{G/M}}(A, \Phi^MX)=\Hom_{R_G}(\infl A, \colim_{V^M=0}\Sigma^V X ).$$
If in addition $A$ is small (for example if it is a sphere), this is 
$\colim_{V^M=0}\Hom_{R_G}(\Sigma^{-V}\infl A, X).$
\end{prop}

\begin{proof}
For brevity, we write $X\tensor \siftyV{M}=\colim_{V^M=0}\Sigma^V X$. 
First note that $\Phi^M(X \tensor \siftyV{M})=\Phi^MX$; indeed, for subgroups
$H$ containing $M$, we have $(\Sigma^VX)(H)=\Sigma^{V^H}(X(H))=X(H)$, 
and all the maps in the direct system are the identity. Accordingly, 
taking geometric fixed points gives a map 
$$\theta: \Hom_{R_G}(\infl A, X\tensor \siftyV{M})\lra \Hom_{R_{G/M}}(A, \Phi^MX). $$
We consider applying $\theta$ to a map $f:\infl A \lra X \tensor \siftyV{M}$. 
If $\theta f=\Phi^Mf=0$ we can see $f=0$. Indeed, if we evaluate at $H$ 
containing $M$
we see $f(H)=0$ since $f(H)=(\Phi^Mf)(H)$. If $H$ does not contain $M$ then 
the basing map 
$$\diagram
(X\tensor \siftyV{M})(H)\rto_{\hspace*{-8ex}\beta_H^{MH}}^{\hspace*{-8ex}\cong}&
(X\tensor \siftyV{M})(MH)=X(MH),
\enddiagram $$
gives an isomorphism, and the commutative square
$$\begin{array}{ccc}
(\infl A)(MH)&\stackrel{f(MH)=(\Phi^Mf)(MH)}\lra &(X\tensor \siftyV{M})(MH)\\
\uparrow &&\uparrow \cong\\
(\infl A)(H)&\stackrel{f(H)}\lra &(X\tensor \siftyV{M})(H)
\end{array}$$ 
shows $f(H)=0$.

To see $\theta$ is an epimorphism, we note that $\Phi^M$ is the 
identity on subgroups $H$ containing $M$, so we may view the 
problem as that of extending  a map
$g: A\lra \Phi^MX$
to a map $f: \infl A \lra X\tensor \siftyV{M}$. 
Now if $H$ is a subgroup not containing $M$ the above commutative 
square shows that we must take $f(H)$ to be the composite
$$(\infl A )(H)\lra (\infl A)(MH)\stackrel{g}\lra X(MH)=
(X\tensor \siftyV{M})(MH). $$
The required compatibility of basing maps associated to an inclusion 
$K \subseteq H$ comes from that of $g$ for $MK \subseteq MH$.
\end{proof}

\section{The associated extended functor}
\label{sec:Gammav}

The purpose of this section is to give a construction of a functor
$\Gamma_v$ replacing an $R$-module by an extended $R$-module, so that
its vertical structure maps become extensions of scalars. 

\begin{thm} 
There is a right adjoint $\Gamma_v=k^!$  to the inclusion
$$\eRmod \stackrel{k}\lra \Rmod. $$
\end{thm}

We will give an explicit construction of the functor $k^!$. 
The main complication is notational, so we begin with two examples.

\subsection{Examples.}
When $G$ is the circle group, the ring is 
$$\diagram
&\cOcFG\dto^v\\
\cOcF \rto^h &\cEi_G \cOcFG
\enddiagram, $$
An arbitrary module takes the form
$$\diagram
&V\dto^{\alpha}\\
P \rto_{\beta} &Q. 
\enddiagram $$
This induces a diagram
$$\diagram
&\cEi \cOcF \tensor V\dto^{\alphat}\\
P \rto_{\beta} &Q,  
\enddiagram $$
and the associated extended module is the pullback
$$\beta': P'\lra \cEi \cOcF \tensor V.$$

Turning to a group of rank 2, we know that for every circle subgroup $K$, 
the construction of $k^!$ on the $G/K$ part of the diagram must do as 
above. Using the $\alphat$ maps  we start with a diagram
$$\diagram
       &                    & \cEi_G    \cOcF\tensor_{\cOcFG}M(G/G)_{G/G} \dto\\
       &\cEi_K \cOcF \tensor_{\cOcFK}M(G/K)_{G/K} \rto \dto &
\cEi_{G}\cOcF\tensor_{\cOcFK} M(G/G)_{G/K} \dto\\
M(G/1)_{G/1} \rto &M(G/K)_{G/1} \rto      
&M(G/G)_{G/1}\\
\enddiagram
$$
Now apply the rank 1 construction on the top two rows. We may then omit
$M(G/G)_{G/G}$ in the top row since the $(G/G)_{G/K}$ entry is extended from 
it. This gives us the $\alpha$-diagram
$$\diagram
           &M(G/K)_{G/K}' \rto \dto &\cEi_{G/K}\cOcFK\tensor_{\cOcFG} M(G/G)_{G/G} \dto\\
M(G/1)_{G/1} \rto &M(G/K)_{G/1} \rto      
&M(G/G)_{G/1}.
\enddiagram
$$
Extending scalars, we obtain the $\alphat$-diagram
$$\diagram
           &\cEi_K \cOcF \tensor_{\cOcFK}M(G/K)_{G/K}' \rto \dto &
\cEi_{G}\cOcF\tensor_{\cOcFG} M(G/G)_{G/G} \dto\\
M(G/1)_{G/1} \rto &M(G/K)_{G/1} \rto      
&M(G/G)_{G/1}.
\enddiagram
$$

Remembering that there are in fact infinitely many circle subgroups $K$, 
we take the pullback of the resulting diagram to give $M(G/1)_{G/1}'$, and
the resulting extended module is 
$$\diagram
M(G/1)_{G/1}'\rto &\cEi_K \cOcF \tensor_{\cOcFK}M(G/K)_{G/K}' \rto  
&\cEi_{G}\cOcF\tensor_{\cOcFG} M(G/G)_{G/G}.  
\enddiagram$$

\subsection{The construction.}
We are ready to give the construction in arbitrary rank.
\begin{defn}
Given a diagram $M$ with $G$ of rank $r$, the construction of the 
associated extended module $k^!M$ proceeds in $r$ steps with 
$$M=k^!_0M\stackrel{\lambda_1}\lla k^!_1M \stackrel{\lambda_2}
\lla \cdots \stackrel{\lambda_r}\lla k^!_rM=k^!M.$$
Assuming $k_n^!M$ has been defined in such a way that the top $n$
vertical maps in each column are extensions of scalars, 
we define $k_{n+1}^!M$ to 
agree with $k_n^!M$ except in the rank $n+1$ row. To fill in the rank $n+1$ 
row, we work from the rank 0 column back towards the rank $n+1$ column. 

We start in the rank 0 column with 
$$k_{n+1}^!M (G/G)_{G/L} :=\cEi_{G/L}\cOcFL\tensor M(G/G)_{G/G}. $$
To fill in the $(G/K)_{G/L}$ entry,  where $\rank (G/L)=n+1$,  we suppose
that the $(G/H)_{G/L}$ entries with $K\subsetneq H$ are already filled in. 
We then obtain a diagram $\Delta_{n+1}M$, with two rows (row 0 and row 1). 
In each row there are entries  $(G/H)_{G/L}$ for $H\supseteq L$. 
Row 1 is the $(n+1)$st row of $k_n^!M$:
$$\Delta_{n+1}M(G/H,1)_{G/L}=k_n^!M(G/H)_{G/L}. $$
Row 0 is obtained from the rank $n$ row of $k_n^!M$. For each 
$H\supsetneq L$ we take
$$\Delta_{n+1}M(G/H,0)_{G/L}=\cEi_{H/L}\cOcFL\tensor_{\cOcFH}k_n^!M(G/H)_{G/H},  $$
noting that, since the $\alphat$-structure maps of $k_n^!M$ above the
rank $n+1$ row are already isomorphisms, this is also the extension of
$k_n^!M(G/H)_{G/K}$ for any $K$ between $L$ and $H$.
Finally, we take 
$$k_{n+1}^!M(G/K)_{G/L}=\ilim \Delta_{n+1}M. $$
\end{defn}

\begin{lemma}
The maps $\lambda_{n+1}: k_{n+1}^!M \lra k_{n}^!M$ induce isomorphisms
$$(\lambda_{n+1})_*:\Hom (L,k_{n+1}^!M) \lra \Hom (L,k_{n}^!M )$$
for any extended $R$-module $L$. In particular $k^!$ is right adjoint
to the inclusion 
$$i: \eRmod\lra \Rmod.$$
\end{lemma}

\begin{proof}
To see $(\lambda_{n+1})_*$ is an epimorphism, suppose $f: L\lra k_n^!M$ is a map, and 
we attempt to lift it to a map $f': L\lra k_{n+1}^!M$ is a map. Of course
we take $f$ to agree with $f'$ except in the rank $n+1$ row, and we use
$f(G/G)_{G/G}$ to give the map on the rank 0 column. We then work along
the rank $n+1$ row; when we come to define $f'(G/K)_{G/L}$ we suppose 
that $f'(G/H)_{G/L}$ has already been defined when $K\subsetneq H$. 
Since $k_{n+1}^!M (G/K)_{G/L}$ is defined as an inverse limit, we use its universal
property. 

To see $(\lambda_{n+1})_*$ is a monomorphism, suppose the two maps $f_1,f_2:L \lra
k_{n+1}^!M$ give the same map to $k_n^!M$.  Evidently
$f_1$ and $f_2$ agree except perhaps in the rank $n+1$ row, and 
we must check they agree there. To start with they agree at $(G/G)_{G/L}$, 
and we then work along the rank $n+1$ row.  When we come to $(G/K)_{G/L}$ 
we observe that $f_1$ and $f_2$ already agree at
$(G/H)_{G/L}$ when $K\subsetneq H$. 
Since $k_{n+1}^!M (G/K)_{G/L}$ is defined as a pullback, the fact that
$f_1$ and $f_2$ agree there follows. 
\end{proof}

\section{Torsion functors}
\label{sec:Gammah}

Various obvious constructions on qce modules (i.e., objects of $\cA (G)$) 
give objects which are extended but not quasi-coherent 
(i.e., objects of $\eRmod=\cAh (G)$). It is therefore 
convenient to have a right adjoint $\Gamma_h$ to 
$$j: \cA (G)=\qceRmod \lra \eRmod=\cAh (G).  $$
Thus $\Gamma_h$ replaces an extended module by a qce module, which is
one in which  the horizontal structure maps are extensions of scalars. 
We refer to as a torsion functor since its restriction to 
objects concentrated at $1$ is the $\cE$-torsion functor.

\begin{thm}
\label{thm:torsionfunctor}
There is a right adjoint  $\Gamma_h : \cAh (G) \lra \cA (G)$
to the inclusion $j: \cA (G) \lra \cAh (G)$. 
\end{thm}

The construction is a natural extension of that given in Chapters 17-20 
of \cite{s1q}, but we will spend less time here describing alternative 
approaches. The construction also works for similar 
inclusions, so for   brevity we  write $j: \cA \lra \cAh$, and allow
ourselves to give examples in rather simpler categories.

\subsection{Motivation}
\label{subsec:GammaMotivation}

 For the purposes of discussion suppose $Y$ is an object of $\cAh $, and consider
the properties $\Gamma_h Y$ is forced to have. Note first that the inclusion $j$
is full and faithful, so we may write $\Hom$ to denote maps both in $\cAh$ and
in $\cA$ without ambiguity. Thus  if $T$ is an object of $\cA (G)$ then
$\Hom(T,Y)=\Hom (T,\Gamma_h Y)$. 

We need only use spheres $S^{-W}$ as test objects; as before we write
$$Y(W)=\Hom (S^{-W},Y).$$
According to Lemma \ref{lem:mapsoutofSV},  this is the
set of elements of $(\Sigma^{-W}Y)(U(1))$ with the same footprint
as the characteristic element of $S^{-W}$. Thus we have 
$$(\Gamma_h Y)(W)=Y(W)$$ 
for all representations $W$.

For example, as in Proposition \ref{prop:leftPhi}, this allows us to 
deduce $\phi^G (\Gamma_h Y)$ by the calculation
$$\begin{array}{rcl}
\phi^G (\Gamma_h Y)&=&\Hom (S^0, (\Gamma_h Y) \tensor S^{\infty V(G)})\\
&=&\colim_{V^G=0}\Hom (S^{-V}, \Gamma_h Y )\\
&=&\colim_{V^G=0}\Hom (S^{-V}, Y) \\
&=&\colim_{V^G=0}Y(V)
\end{array}$$
In fact, the main difference between
$Y$ and $\Gamma_h Y$ is that in $\Gamma_h Y$ the elements with different 
footprints must be more separate. 

\begin{example}
We consider the case of semifree objects for  the circle group, so that
$\cAh$ consists of all maps $\beta: N \lra \Q [c,c^{-1}] \tensor V$, and 
$\cA$ is the subcategory in which the map $\beta $ becomes an isomorphism
when $c$ is inverted (see \cite[Chapter 18]{s1q} for more details). 

The case $Y=(\Q [c] \lra 0)$ is instructive (see \cite[18.3.2]{s1q}).
In this case, the relevant spheres are 
$$S^{-kz}=(\Sigma^{-2k}\Q[c] \stackrel{c^k}\lra \Q [c,c^{-1}] \tensor \Q).$$
Here $\beta (1)=c^k \tensor 1$, so that the footprint is $c^k$. Thus
$$Y(kz)=Y(c^k)=\{ y \in Y(U(1)) \st \beta (y)=c^k \tensor y' \mbox{ for some } y'\}.$$
Thus we see $Y(c^k)=\Q[c]$ for all $k$. On the other hand
 $\phi^G(\Gamma_h Y)=\Q[c,c^{-1}]$ so that, as subspaces of $(\Gamma_h Y)(U(1))$, 
the subspaces $(\Gamma_h Y)(c^k)=Y(c^k)=\Q [c]$ for different $k$ only intersect
in 0. Indeed, we find
$$\Gamma_h Y=(N \lra \Q[c,c^{-1}]\tensor \Q [c,c^{-1}])$$
where 
$$N=\ker (\mu : \Q [c,c^{-1}]\tensor \Q [c,c^{-1}]\lra \Q [c,c^{-1}]).$$
\end{example}

\subsection{Torsion and fixed points}

In describing the torsion functor it is convenient to begin by 
observing some of the properties it will need to have. Accordingly 
we start by assuming the right adjoint exists (we could easily 
establish this by checking the formal 
properties of the inclusion, but since we construct the functor
explicitly, this is unnecessary). 

The following is immediate from Proposition \ref{prop:leftPhi}. 

\begin{lemma}
If $A$ is quasi-coherent then for any extended module $X$ we have 
$$
\Hom_{R_{G/M}}(A, \Phi^M \Gamma_h X)=
\colim_{V^M=0}\Hom_{R_G}(\infl A \tensor S^{-V}, X).\qqed
$$
\end{lemma}

Taking $A=S^0$ and $M=G$ we find the vertex of $\Gamma_h X$.
\begin{cor}
\label{cor:vertixofGamma}
$$(\Gamma_h X)(G/G)_{G/G}=\colim_{V^G=0}\Hom_{R_G}(S^{-V}, X).$$
\end{cor}

\subsection{Internal and external Euler classes.}

If $W, W'$ are complex representations with $W \subseteq W'$, 
then the inclusion
$S^W \lra S^{W'}$ is associated to Euler classes in various ways.

Firstly, the map 
$$ \Sigma^W Y \lra \Sigma^{W'}Y$$
induces multiplication by the Euler class $e(W'-W)$ on evaluation
at the trivial subgroup.

 On the other hand, we can form another version of this
by using internal Euler classes. Thus,
the inclusion $S^W \lra S^{W'}$ of spheres induces a map 
$$ Y(W)=\Hom (S^{-W},Y) \lra \Hom (S^{-W'}, Y)=Y(W') $$

Note that there are maps $Y(W) \lra (\Sigma^W Y)(U(1))$, and 
that the square 
$$\diagram
Y(W) \rto \dto &(\Sigma^W Y)(U(1))\dto\\
Y(W') \rto     &(\Sigma^W Y)(U(1))
\enddiagram$$
commutes. There are therefore two compatible ways to pass to 
limits over diagrams of representations.

By definition, for qce modules we have
$\cEi_K Y(U(1))=\cEi_K \cOcF \tensor_{\cOcFK}\phi^KY$.
The following lemma records the fact that therefore the internal
and external limits agree for qce modules.

\begin{lemma}
If $Y$ is qce, and $W$ is a representation of $G/K$  then 
$$\colim_{V^K=0}\Hom (S^{-(W \oplus V)}, Y)=(\phi^KY)(W).\qqed$$
\end{lemma}

In view of this, it is convenient to write
$$Y(W \oplus \infty V(K)):=\colim_{V^K=0}Y(W \oplus V).$$

\subsection{The construction.}

The idea in building $\Gamma_h Y$ is to build up the geometric 
fixed points $\phi^L \Gamma_h Y$ in order of increasing codimension 
of $L$, referring to a localization diagram to spread out the parts 
with each footprint. 

To start with, as suggested by the discussion in Subsection 
\ref{subsec:GammaMotivation}, we must take 
$$\phi^G (\Gamma_h Y)=(\Gamma_h Y)(\infty V(G)):=Y(\infty V(G)). $$
Now, suppose $c \geq 1$, that  $\phi^K(\Gamma_h Y)$ has been constructed for all 
$K$ of codimension $\leq c -1$ and  that $L$ is of codimension $c$. 

We form $\phi^L(\Gamma_h Y)$ by modifying $\phi^LY$ so as 
to be compatible with the values $\phi^K(\Gamma_h Y)$ for 
$K$ of codimension $\leq c-1$. In fact we use a  pullback square
of the following form
$$\diagram
\phi^L(\Gamma_h Y)\rto \dto &\phi^L Y\dto\\
CH^0(L;\Phi^L \Gamma_h  Y) \rto & CH^0(L;l \phi^LY), 
\enddiagram$$
where the lower horizontal needs to be explained. For the present, 
the important thing is that it is defined using only values on 
subgroups of codimension $\geq c+1$. This defines an extended 
sheaf $\Gamma_h Y$ of $\cO$-modules with a natural transformation 
$\Gamma_h Y \lra Y$.

To explain the lower horizontal, we begin with 
the \Cech\ functor $CH^0(L;F)$. It is defined for a functor $F$ on the 
{\em non-trivial} connected subgroups of $G/L$. First we choose 
subgroups $S_1,S_2, \ldots , S_c$ so that we have 
a direct product decomposition $G/L= S_1 \times S_2 \times \cdots \times S_c$.
Now form the associated \Cech-type complex
$$\CC (L;F)=\left[
\prod_i F(S_i) \lra \prod_{i<j}F(S_i \times S_j) \lra 
\prod_{i<j<k} F(S_i\times S_j \times S_k)\lra \cdots \right].$$
The cohomology of this is $CH^*(L;F)$. We will show in Lemma 
\ref{CHzeroinvariant} below that for the functors
$F$ that concern us,  this  cohomology is independent of the
choice of subgroups. 

Next we describe the two functors $F$ to which we must apply $CH^0$.
By induction, the functor $F=\Phi^L \Gamma_h Y$ is already defined on 
all non-trivial connected subgroups, and is therefore a legitimate
entry. In particular,  the bottom left entry of the pullback square
is defined. The functor $F=l\phi^LY$ is the qc module obtained by 
localization
$$(l\phi^L Y)(K/L):= \cEi_{K/L} \phi^LY, $$
and the bottom right entry of the pullback square is defined.

By construction there is a natural transformation 
$\Phi^L \Gamma_h Y \lra l\phi^LY$  
on subgroups of codimension $\geq c+1$. This gives 
the bottom horizontal in the pullback square. 
Indeed if $L \subseteq K$, the map 
at $K/L$ is  
$$\cEi_{K/L} \cOcFL \tensor_{\cOcFK} \phi^K \Gamma_h Y 
\lra \cEi_{K/L}\phi^L Y.$$
For this we note that for representations $W$ of $G/K$ we have maps
\begin{multline*}
(\phi^K\Gamma_h Y)(W)=(\Gamma_h Y)(W \oplus \infty V(K))
=Y(W\oplus \infty V(K))\\
 =Y(W\oplus \infty V(K/L)\oplus \infty V(K)) 
\lra (\Phi^L Y)(W \oplus \infty V(K/L)) \lra \cEi_{K/L}\Phi^LY(W).
\end{multline*}
By the universal property of localization, we get the required maps.

\subsection{Independence of choices.}

For proofs it is  helpful to note the analogy with well-known 
constructions from commutative algebra. If we have a commutative
ring $R$ and an element $x$ we may form  the stable
Koszul complex 
$$\SK(x)=(R \lra R[1/x]),$$
and given elements $x_1, \ldots , x_c$, and a module $M$, we may form
$$\SK(x_1, \ldots ,x_c;M)=
\SK(x_1)\tensor \cdots \tensor \SK(x_c)\tensor M.$$
The \Cech\ complex is obtained by deleting the $M$ in degree 0 and
shifting degree, so that there is a fibre sequence
$$\SK(x_1, \ldots, x_c;M) \lra M \lra \CC (x_1, \ldots , x_c;M).$$
Indeed, if we replace elements by multiplicative sequences, we find
$$\CC (L; l\phi^LY)=\CC (\cE_{S_1}, \ldots , \cE_{S_c}; \phi^L Y).$$
We have the analogue of the well know fact that \Cech\ cohomology 
is geometric.

\begin{lemma}
\label{CHzeroinvariant}
For the two functors $F=\phi^L \Gamma_h Y$ and $F=l \phi Y$, the 
cohomology $CH^0(L;F)$ is independent of the choice of subgroups
$S_1, \ldots, S_c$.
\end{lemma}

\begin{proof}
This is the analogue of the fact that the \Cech\ complex for a
sequence of elements depends only on the ideal they generate.
An elementary proof begins by observing that 
if $x \in (x_1, \ldots, x_c)$ then the natural map  
$$\CC (x_1, \ldots , x_c, x;M) \lra \CC (x_1, \ldots , x_c;M) $$
is a chain equivalence. Indeed, providing we use the right
proof, it applies in our situation. The point is to recognize
the fibre of this map as
$$\SK (x_1, \ldots , x_c; M[1/x]).$$

Now apply this to our context, assuming that $S_{c+1}$ is another
circle intersecting each of $S_1, \ldots , S_c$ only in the identity.
From the commutative algebra,  we see that the fibre of 
$$\CC (L;F)_{S_1,\ldots, S_c}\lra \CC (L;F)_{S_1,\ldots, S_{c+1}}$$
consists of a stable Koszul complex formed from values of  $F$ 
on which it is already defined. Furthermore, the multiplicative set
 $\cE_{S_{c+1}}$ is inverted on all of those values, and  
every element of this multiplicative set
is in the ideal generated by the multiplicative sets
$\cE_{S_1}, \ldots , \cE_{S_c}$. The fibre is therefore acyclic.

To complete the proof, we can move between any two sets of circles giving $G$ as 
a product by adding and removing circles in general position.
\end{proof}

\subsection{Properties}
It remains to check that the construction has the desired properties. 

\begin{lemma}
If $Y$ is quasi-coherent then $\Gamma_h Y=Y$.
\end{lemma}

\begin{proof}
Again we prove this at $L$ by induction on the codimension of $L$. When $L=G$
we have $\phi^G Y =Y(\infty V(G))$ as required.

Now suppose $L$ is of codimension $c \geq 1$ and 
$\phi^K\Gamma_h Y \lra \phi^K Y$ is an isomorphism for all
$K$ of codimension $\geq c+1$. It follows that the lower horizontal
in the defining pullback square  is an isomorphism, and hence the upper 
is an isomorphism  as required. 
\end{proof}

\begin{lemma}
For any $Y$, the module $\Gamma_h Y$ is quasicoherent.
\end{lemma}

\begin{proof}
We need to show that 
$\cEi_{K/L}\phi^L\Gamma_h Y=\cEi_{K/L}\cOcFL \tensor_{\cOcFK}\phi^K (\Gamma_h Y)$
when $L$ is of codimension 1 in $K$. For notational simplicity
we treat the case $L=1$.

Choose $S_1=K$ in the decomposition of $G$.  
Consider the diagram obtained from the defining pullback square
by inverting $\cE_K$. First note that 
$$\cEi_K Y(U(1))=\cEi_K CH^0(L,l Y(U(1)));$$
this is most easily seen by noting that the augmented
complex $Y(U(1)) \lra \CC (L;lY)$ is the stable Koszul complex
$$\SK (\cE_K, \cE_{S_2}, \ldots , \cE_{S_c};Y(U(1))).$$ 
It is therefore a tensor product of stable
Koszul complexes, the first of which,  $Y(U(1)) \lra \cEi_{K}Y(U(1))$, 
becomes acyclic when we invert $\cE_K$.

From the properties of pullbacks, the left hand vertical becomes an isomorphism
when $\cE_K$ is inverted.  We therefore need to show that 
$$\cEi_{K}CH^0(L;\phi^L \Gamma_h Y)=\cEi_L \cOcF \tensor_{\cOcFK} \phi^K \Gamma_h Y$$
is an isomorphism. We want to apply the same argument, but we need to 
avoid reference to the value of $\Gamma_h Y$ at $1$, which has not yet
been defined. For example we may pick off each of the
 quotient complexes
$$\cEi_K (\cEi_{S_j}\cOcF \tensor_{\cOcFSj} \phi^{S_j}\Gamma_h Y )\lra 
\cEi_{K \times S_j}\cOcF \tensor_{\cOcFKSj} \phi^{K \times S_j}\Gamma_h Y$$
for $j \neq 1$, which are acyclic by induction, until we are left with just
$$\cEi_K (\cEi_{K}\cOcF \tensor_{\cOcFK} \phi^{K}\Gamma_h Y)= 
\cEi_{K}\cOcF \tensor_{\cOcFK} \phi^{K}\Gamma_h Y$$ 
in degree 0.
\end{proof}

We thus have a unit $\Gamma_h j X \lra X$ which is an isomorphism and
counit $j \Gamma_h Y \lra Y$, which is an isomorphism on $\cA$. It follows
that the functor $\Gamma_h $ is right adjoint to inclusion.

\end{document}